\documentclass{amsart}

\DeclareMathSymbol{\twoheadrightarrow}  {\mathrel}{AMSa}{"10}
        
\def\HH{{\mathcal H}}

\def\C{{\mathbb C}}

                          \def\AA{{\mathbb A}}

\def\P{{\mathbb P}}

\def\f{{\tilde F}}
                     \def\f0{{\mathfrak f}}

\def\A8{{\mathbf A}_8}
\def\Bir{\mathrm{Bir}}

                       \def\proj{\mathrm{proj}}

\def\Aut{\mathrm{Aut}}

\def\II{{{\mathrm I}}}

              \def\L{{\mathcal L}}

\def\fchar{\mathrm{char}}
             
\def\GL{\mathrm{GL}}

\def\A{\mathcal{A}}
\def\C{\mathcal{C}}

\def\T{{\mathcal T}}
     \def\t{\mathrm{t}}
\def\B{{\mathcal B}}

                         \def\TT{\mathrm{T}}
\def\dim{\mathrm{dim}}

       \def\PGL{\mathrm{PGL}}

       \def\G{{\mathfrak G}}

          \def\l1{{\mathbf 1}}

                                                     \def\HH{{\mathcal H}}
                         \def\AA{\mathbb{A}}

\newtheorem{thm}{Theorem}[section]
\newtheorem{lem}[thm]{Lemma}
\newtheorem{cor}[thm]{Corollary}
\newtheorem{prop}[thm]{Proposition}
\theoremstyle{definition}
\newtheorem{defn}[thm]{Definition}

\newtheorem{rem}[thm]{Remark}

\newtheorem{rems}[thm]{Remarks}
        \newtheorem{sect}[thm]{}

\begin{document}

\title[Jordan groups]{Jordan groups and elliptic ruled surfaces}

\author{Yuri G. Zarhin}
\thanks{This work was partially supported by a grant from the Simons Foundation (\# 246625 to Yuri Zarkhin)}
\address{Department of Mathematics, Pennsylvania State University, University Park, PA 16802, USA}
\address{Department of Mathematics, The Weizmann Institute of Science, POB 26,  Rehovot 7610001, Israel}
\email{zarhin@math.psu.edu}

\dedicatory{In memoriam of Emmanuil El'evich Shnol}

\begin{abstract}
We prove that an analogue of Jordan's theorem on finite subgroups of general linear groups  holds
 for the groups of biregular automorphisms of elliptic ruled surfaces. This gives a positive answer to a question of Vladimir L. Popov.
\end{abstract}


\maketitle

\section{Introduction}

We write $k$ for an algebraically closed field, $\AA^1$, $\AA^2$ and $\P^1$ for the affine line, affine plane and projective line respectively (all over $k$).
If $U$ is an irreducible algebraic variety over $k$ then $k[U]$, $k(U)$, $\Aut(U)$ and $\Bir(U)$ stand for its ring ($k$-algebra) of regular functions, its field of rational functions, its group of biregular automorphisms and the group of birational $k$-automorphisms respectively. If $E$ is a vector bundle of finite rank over $U$ then we write  $\Aut(E)$ for the group of (biregular) automorphisms of $E$ (that leave invariant every fiber and act linearly on it).  If $z$ is a $k$-point on $U$ then we write $E_z$ for the fiber of $E$ over $z$; if $u \in \Aut(E)$ then we write $u_z$ for the linear automorphism of the $k$-vector space $E_z$ induced by $u$. As usual, we say that a rank $2$ vector bundle $E$ over $U$ is decomposable (respectively, indecomposable) if it is isomorphic (respectively, not isomorphic) to a direct sum of two line bundles over $U$.

If $Y$ is an abelian variety over $k$ and $z$ is a $k$-point of $Y$ then we call the {\sl  translation} by $z$ the biregular automorphism $\TT_z$ of $Y$ defined by the formula
 $$\TT_z:y \mapsto y+z.$$
There is the natural embedding
$$Y(k) \hookrightarrow \Aut(Y), \ z \mapsto \TT_z. \eqno{(0)}$$
Further we identify $Y(k)$ with its image in $\Aut(Y)$. It is well known that if $\dim(Y)=1$ then $Y(k)$ is a subgroup of finite index in $\Aut(Y)$.

We write $\II_U$ for the {\sl trivial} line bundle $U\times \AA^1$ over $U$ and $\II^2_U$ for the {\sl trivial} rank $2$ vector bundle
$$\II_U^2=\II_U \oplus \II_U=(U\times \AA^1) \times_U (U\times \AA^1)=U\times \AA^2$$
over $U$. By an elliptic curve we mean an irreducible smooth projective curve of genus 1. If $A$ is a finite set (group) then we write $\#(A)$ for its number of elements (order).

The following definition was inspired by the classical theorem of Jordan \cite[Sect. 36]{CR} about finite subgroups of general linear groups (over fields of characteristic zero).

\begin{defn}[Definition 2.1 of \cite{Popov}]
A group $B$ is called a {\sl Jordan group} if there exists a positive integer $J_B$ such that every finite subgroup $B_1$ of $B$ contains a
normal commutative subgroup, whose index in $B_1$ is at most $J_B$.
\end{defn}

\begin{rem}
Clearly,  a subgroup of a Jordan group is also Jordan. If a Jordan group  $G_1$ is a subgroup of {\sl finite} index in a group $G$ then $G$ is also Jordan.
 Every commutative group is Jordan and all finite groups are Jordan.  A product of two Jordan groups is also Jordan.
See \cite[Sect. 1]{Popov2} also for plenty of examples  of Jordan (and non-Jordan) groups. See also \cite{Tu,Popov3,ProSh}.
\end{rem}

 V. L. Popov (\cite[Sect. 2]{Popov}, see also \cite{Popov2}) posed a question whether $\Aut(S)$ is a Jordan group when $S$ is an irreducible algebraic surface over $k$ with $\fchar(k)=0$.
 He obtained a positive answer to his question for almost all
 surfaces. (The case of rational surfaces was treated earlier by J.-P. Serre \cite[Sect. 5.4]{Serre}.) The only remaining case is  when $S$ is birationally (but not biregularly) isomorphic to a product $X\times\P^1$ of an elliptic curve $X$ and the projective line.

In what follows we always assume that $\fchar(k)=0$.
 Our main result is the following statement, which gives a positive  answer to Popov's question.

 \begin{thm}
 \label{elliptic}
 If $X$ is an elliptic curve over $k$  and $S$ is an irreducible normal projective algebraic surface that is birationally isomorphic to $X\times\P^1$ then
 $\Aut(S)$ is a Jordan group.
 \end{thm}

\begin{rem}
 It is known \cite{ZarhinEdinburgh}
 that $\Bir(X\times\P^1)$  is {\sl not} a Jordan group
 \end{rem}

 \begin{rem}
 Suppose that $S$ is a {\sl non}-smooth normal surface. Since it is normal, there are only finitely many  singular points on $S$. Then, by \cite[Sect. 2, Cor. 8]{Popov2}, $\Aut(S)$ is Jordan.
 This implies that in the course of the proof of Theorem \ref{elliptic} we may assume that $S$ is smooth.
 \end{rem}

 \begin{cor}
 \label{surfaceJ}
 Suppose that $V$ is an irreducible normal projective algebraic variety over $k$. If $\dim(V) \le 2$ then $\Aut(V)$ is Jordan.
 \end{cor}

 \begin{proof}[Proof of Corollary \ref{surfaceJ}]
 We have $\Aut(V) \subset \Bir(V)$. If $V$ is {\sl not} birationally isomorphic to a product of the projective line and an elliptic curve then
 $\Bir(V)$ is Jordan (\cite[Th. 2.32]{Popov}) and therefore its subgroup $\Aut(V)$ is also Jordan. If $V$ is birationally isomorphic to a product of the projective line and an elliptic curve then $\dim(V)=2$ and Theorem \ref{elliptic}  implies that $\Aut(V)$ is Jordan.
 \end{proof}

\begin{thm}
\label{NonNormal}
Let $V$ be an irreducible projective algebraic variety over $k$. If $\dim(V) \le 2$ then $\Aut(V)$ is Jordan.
\end{thm}

\begin{proof}[Proof of Theorem \ref{NonNormal}]
Let $\nu:V^{\nu} \to V$ be the {\sl normalization} of $V$ (\cite[Ch. II, Sect. 5]{ShAG},  \cite[Ch. III, Sect. 8]{MumfordRed}).  Here $\nu$ is a birational regular map and
 $V^{\nu}$ is an irreducible  {\sl normal projective} variety of the same dimension (as $V$) over $k$ \cite[Th. 4 on p. 203]{MumfordRed}.
It is well known  that every biregular automorphism of $V$ lifts uniquely to a biregular automorphism of $V^{\nu}$ \cite[Ch. 2,  Sect. 2.14, Th. 2.25 on p, 141]{Itaka}. This
give rise to the {\sl embedding} of groups
$$\Aut(V)\hookrightarrow \Aut(V^{\nu}).$$
By  Corollary \ref{surfaceJ}, the group $\Aut(V^{\nu})$ is Jordan. Since $\Aut(V)$ is isomorphic to a subgroup of  Jordan   $\Aut(V^{\nu})$, it is also Jordan.
\end{proof}

\begin{cor}
\label{FinalProj}
Let $V$ be a projective algebraic variety over $k$ and $\Aut(V)$ the group of biregular automorphisms of $V$.
 If $\dim(V)\le 2$ then $\Aut(V)$ is Jordan.
\end{cor}
\begin{proof}
Let $V_1, \dots , V_r$ be all the {\sl irreducible} components of $V$. Clearly, all $V_i$ are irreducible projective varieties with
$\dim(V_i) \le \dim(V) \le 2$. By Theorem  \ref{NonNormal}, all $\Aut(V_i)$ are Jordan. Now Lemma 1 in Section 2.2 of \cite{Popov2} implies that $\Aut(V)$ is also Jordan.
\end{proof}

\begin{rem}
Using Corollary \ref{FinalProj}, T.M. Bandman and the author recently have proven that the group $\Aut(V)$ is Jordan for all algebraic surfaces $V$ \cite{BaZar}.
\end{rem}

The paper is organized as follows. Section \ref{uniruled} contains preliminary remarks about smooth projective uniruled surfaces that are fibered over an elliptic curve. In Section \ref{Bir1} we discuss interrelations between the automorphisms of elliptic ruled surfaces and rank 2 vector bundles over elliptic curves. Our exposition is based on beautiful results of M. Maruyama \cite{Mar1,Mar2}.
In Section \ref{theta} we discuss Mumford's theta groups \cite[Sect. 23]{Mumford}  viewed as automorphism groups of certain {\sl decomposable} elliptic ruled surfaces. In Section \ref{final} we prove the main result. Section \ref{group} contains auxiliary results from group theory that were used in Section \ref{Bir1}.

{\bf Acknowledgements}. I am deeply grateful to Volodya Popov for a stimulating question and useful discussions,  and to Tatiana Bandman for useful discussions. My special thanks go to the referees, whose comments helped to improve the exposition.

This work was done during the academic year 2013/2014 when I was Erna and Jakob Michael Visiting Professor in the Department of Mathematics  at the Weizmann Institute of Science, whose hospitality and support are gratefully acknowledged.

\section{Uniruled surfaces}
\label{uniruled}
Let $X$ be an elliptic curve over $k$ and let $S$ be an irreducible smooth projective surface over $k$  that is birationally isomorphic to $X\times\P^1$.  Clearly, there is plenty of rational sections $X \to S$ of $\pi$. Since $X$ is a smooth curve and $S$ is projective, each such a section extends to the regular map $X \to S$, which is a section of $\pi$.

\begin{rems}
 \label{prelim}
 \begin{itemize}
 \item[(i)]
 Let us fix a birational isomorphism between {\sl smooth} projective $S$ and $X\times\P^1$. The projection map $X\times\P^1\to X$ gives rise to a rational map $\pi: S \to X$ with dense image. Since $S$ is smooth and $X$ becomes an abelian variety (after the choice of a base point), it follows from a theorem of Weil \cite[Sect. 4.4]{Neron} that $\pi$ is regular. Since $S$ is projective,  $\pi: S \to X$ is surjective, because its image is closed.  Clearly, $X$ may be identified (after the choice of a base point) with the Albanese variety of $S$ and $\pi$ is the universal Albanese map for $S$ \cite{ShS}. In particular, every biregular automorphism $\sigma$ of $S$ induces (by functoriality) a certain biregular automorphism $f(\sigma)$ of $X$. This gives rise (see \cite[Lemma 6 on p. 97]{Mar2})) to an exact sequence
 $$\{1\} \to \Aut_X(S) \subset \Aut(S) \stackrel{f}{\to} \Aut(X) \eqno{(1)}$$
 where the subgroup
 $\Aut_X(S)$ consists of all biregular automorphisms $\sigma \in \Aut(S)$ such that $\pi\sigma=\pi$
 (i.e. $\sigma$ leaves invariant every fiber of $\pi$)
 and the group homomorphism
 $$f:  \Aut(S) \to \Aut(X)$$
 is characterized by the property
 $$f(\sigma)(\pi(s))=\pi(\sigma(s)) \ \forall \sigma \in \Aut(S), s \in S.$$
 (In other words, $\Aut_X(S)$ consists of all biregular automorphisms of $S$ that induces the identity map on the base $X$.)
 For each $z \in X(k)$ the automorphism $\sigma \in \Aut_X(S)$ induces the biregular automorphism of the fiber $\pi^{-1}(z)$, which we denote by $\sigma_z$.

 Clearly, the generic fiber of $\pi$ is the projective line $\P^1_{k(X)}$ over the field $k(X)$ of rational functions of $X$. This gives us an embedding
 $$\Aut_X(S) \hookrightarrow \Aut(\P^1_{k(X)})=\PGL(2,k(X)).$$
 The group  $\PGL(2, k(X))$ is a {\sl linear group} that is isomorphic (via the {\sl adjoint} representation) to a subgroup of $\GL(3 ,k(X))$. It follows  from the theorem of Jordan that $\PGL(2, k(X))$ is Jordan.
 This implies that its subgroup $\Aut_X(S)$ is also Jordan. It follows that if the image
 $$f( \Aut(S))\subset \Aut(X)$$
 is finite then $\Aut(S)$ is also Jordan.

 So, in the course of the proof of Theorem \ref{elliptic} we may assume that $f(\Aut(S))$ is infinite.

 \item[(ii)]
 Let $x_0 \in X(k)$ be a $k$-point of $X$. Then one may define on $X$ the structure of a one-dimensional abelian variety, by taking $x_0$ as the zero of the group law.
 The subgroup
 $$\Aut_{x_0}(X)=\{u \in \Aut(X)\mid u(x_0)=x_0\}$$
 is the automorphism group of the one-dimensional abelian variety $X$ and therefore is finite. On the other hand, if $\T \subset X(k)$ is a {\sl nonempty} finite set of $k$-points on $X$ then the (sub)group
 $$\Aut(X,\T) =\{u \in \Aut(X)\mid u(\T)=\T\} \subset \Aut(X)$$
 contains a subgroup of finite index that lies in $\Aut_{x_0}(X)$ for every $x_0 \in \T$. This implies that $\Aut(X,\T)$ is finite.
 \end{itemize}
 \end{rems}

\begin{rem}
\label{prelimiv}
 Let $S_0$ be the (finite) set of all {\sl degenerate} fibers of $\pi$, i.e.,  the fibers $\pi^{-1}(x)$  (with $x\in X(k)$)  that are {\sl not} isomorphic (as a closed subscheme) to $\P^1$.
 Suppose that $S_0$ is {\sl not} empty and therefore its image $\T:=\pi(S_0)$ is a finite nonempty subset of $X(k)$. Clearly, $\Aut(S)$ permutes elements of $S_0$ and therefore $f(\Aut(S))$ permutes the elements of $\T$. This implies that
 $$f(\Aut(S))\subset \Aut(X,\T).$$
 Since $\T$ is nonempty and finite, the group $\Aut(X,\T)$ is finite. This implies that $\Aut_X(S)$ is a subgroup of finite index in $\Aut(S)$. Since $\Aut_X(S)$ is Jordan (Remark \ref{prelim}(i)), $\Aut(S)$ is also Jordan.

  It follows that in the course of the proof of Theorem \ref{elliptic} we may assume that all the fibers of $\pi$ are biregularly isomorphic to $\P^1$, i.e., $\pi: S \to X$ is a $\P^1$-bundle. In other words,  we may assume that $S$ is an {\sl elliptic ruled} surface.
\end{rem}

\section{Elliptic ruled surfaces}
\label{Bir1}

Let $\pi: S \to X$ be an elliptic ruled surface, i.e., a $\P^1$-bundle over an elliptic curve $X$. It is known \cite{ShS} that there is a rank $2$ vector bundle over $X$ such that $S$ is biregularly isomorphic to the projectivization $\P(E)$ of $E$ in such a way that $\pi$ becomes the corresponding canonical map
$$\pi: \P(E) \to X.$$
By definition, for each $z \in X(k)$ the fiber $\pi^{-1}(z)$ is the projectivization $\P(E_z)$ of the two-dimensional $k$-vector space $E_z$.

If $\L$ is a line bundle on $X$ then the tensor product $E\otimes \L$ is also  rank $2$ vector bundle over  $X$ and their projectivizations $\P(E)$ and $P(E\otimes \L)$ are canonically isomorphic as $\P^1$-bundles over $X$. Since $E$ is locally trivial, every point $x \in X$ admits an open Zariski neighbourhood $U_x \subset X$ such that the preimage $\pi^{-1}(U_x)$ is biregularly isomorphic to a direct product $U_x \times \P^1$ and the map $\pi^{-1}(U_x) \stackrel{\pi}{\to} U_x$ corresponds (under this isomorphism) to the projection map
$U_x \times \P^1 \to U_x$.

There is a short  exact sequence  \cite[p. 94]{Mar2}
$$\{1\} \to k^{*}\to \Aut(E) \stackrel{\proj}{\to} \Aut_X(S) \to \varDelta \to \{1\} \eqno{(2)}.$$
Here each $\alpha \in k^{*}$ acts as (the corresponding  {\sl homothety} of $E$, i.e., as) multiplication by  $\alpha$ in each fiber $E_x$ (and, of course, induces the identity map on the projectivization), for each $z \in X(k)$ and $u \in \Aut(E)$ the automorphism
 $\proj(u)_z$ of the projective line
$\pi^{-1}(z)=\P(E_x)$  is the {\sl projectivization} of  $u_z\in \Aut(E_z)$. The group $\varDelta$ is a certain finite commutative group that is either trivial or has exponent $2$.

Recall   \cite{Atiyah} that  $E$ has a line subbundle.
There is a natural bijection between regular sections of $\P(E) \to X$ and line subbundles of $E$ \cite[Sect. 3]{Mar1}.
  Every section $s:X \to S=\P(E)$ of $\pi$  gives rise to the line subbundle $E(s)\subset X$. Namely, for each $x \in X(k)$ the fiber of $E(s)$ over $x$ is the one-dimensional subspace of the two-dimensional fiber $E_x$ of $E$ that corresponds to $s(x)\in \P(E_x)$. Conversely, a line subbundle $L \subset E$ gives rise to the section $t_L: X \to S$
where the fiber $L_x$ is the one-dimensional subspace $L_x \in \P(E_x)$. Clearly,
$$E(t_L)=L, \ t_{E(s)}=s.$$
If $L$ is a line subbundle in $E$ then $L\otimes\L$ is a line subbundle in $E\otimes \L$ and
$$t_L=t_{L\otimes\L}.$$
If $s=t_L$ then its image is an (irreducible) effective divisor on $S$, whose self-intersection index
$$(s\cdot s)=\deg(E)-2\deg(L)$$
(see \cite[pp. 17--18]{Mar1}).

A line subbundle $L\subset E$ is called {\sl maximal} if its degree is maximal among the degrees of all line subbundles of $E$. It is known \cite[Sect. 1, p. 6 and Sect. 3]{Mar1} that a maximal line subbundle does exist.   If $L$ is maximal then the corresponding section $s=t_L$ is called {\sl minimal}. A section $s$ is minimal if and only if $(s\cdot s)$ is the smallest among the self-intersection indices of all sections of $\pi$ (see \cite[p. 18--19, Th. 1.16]{Mar1}. If $L$ is maximal then we put
$$N(E)=\deg(E)-2\deg(L).$$
It is known  \cite[p. 11, Prop. 1.9]{Mar1}  that
$$N(E)=N(E\otimes\L).$$
This allows us to introduce the notation
$$N(S):=N(E).$$
It is known that if $N(E)>0$ then $E$ is indecomposable \cite[p. 15, Proof of Cor. 1.12]{Mar1}).

\begin{sect}
\label{maruyama}
We will need the following results of Maruyama \cite{Mar1,Mar2}.

\begin{itemize}
\item[(i)]
If $N(S)=N(E)>0$ then
$\Aut_X(S) \cong \varDelta$ is a {\sl finite} group (see \cite[p. 95, Th. 2(1)]{Mar2}).

\item[(ii)]
If $E$ is {\sl decomposable} and $N(E) \ne 0$ then the image $f(\Aut(S)) \subset \Aut(X)$ in the exact sequence (1) is a {\sl finite} group (see Lemma 7 of \cite[p. 98]{Mar2}.

\item[(iii)]
If $E$ is {\sl indecomposable} and $N(E) \le 0$ then there is a nonnegative integer $r$ such that the group $\Aut_X(S)$ is isomorphic to the direct sum of $r$ copies of the additive group $k$.
(See Theorem 2(2) of \cite[pp. 95--96]{Mar2}.)

\item[(iv)]

Suppose that $E$ is  a direct sum $L_1\oplus L_2$  of  line bundles
$L_1$ and $L_2$. Clearly, the images of $t_{L_1}$ and $t_{L_2}$ in $S=\P(L_1\oplus L_2)$ do not meet each other.

If $L$ is a line subbundle of $E$ then $\deg(L)\le \max (\deg(L_1),\deg(L_2))$ (see \cite[Proof of Lemma 1.1 on p. 6]{Mar1}.)  Assume  that  $\deg(L_1)=\deg(L_2)$.  Then  both $L_1$ and $L_2$ are {\sl maximal} in $E$ and
$$\deg(E)=\deg(L_1)+\deg(L_2)=2\deg(L_1)=2\deg(L_2).$$
 It follows that
$$N(E)=0.$$
 If, in addition, $L_1$ and $L_2$ are {\sl not} isomorphic then the set of all maximal line subbundles in $E$ consists of $L_1$ and $L_2$. (See \cite[Proof of Cor. 1.6(ii)  on p. 9]{Mar1} and Lemma 2(2) of \cite[p. 92]{Mar2}.)
Therefore the set of all {\sl minimal} sections of $S=\P(E)=\P(L_1\oplus L_2) \stackrel{\pi}{\to} X$  consists of $t_{L_1}$ and $t_{L_2}$.
\end{itemize}

\end{sect}

\begin{lem}
\label{Npositive}
Suppose that $N(E)>0$. Then $\Aut(S)$ is a Jordan group.
\end{lem}

\begin{proof}
Since $N(E)>0$,  Sect. \ref{maruyama}(i) tells us that
$\Aut_X(S)$ is a finite group. Using the exact sequence (1) (Remark \ref{prelim}(i)), we obtain that $\Aut(X)$ sits in an exact sequence
$$\{1\} \to \varDelta \to \Aut(S) \stackrel{f}{\to}  \Aut(X).$$
Since $\varDelta$ is finite, Corollary \ref{abFinite} (see below) tells us that $\Aut(S)$ is Jordan.
\end{proof}

\begin{lem}
\label{DnotZero}
Suppose that $E$ is decomposable and $N(E) \ne 0$. Then $\Aut(S)$ is a Jordan group.
\end{lem}

\begin{proof}
By  Sect. \ref{maruyama}(ii), the image $f(\Aut(S)) \subset \Aut(X)$ in the exact sequence (1) is finite. Now the result follows from Remark \ref{prelim}(i).
\end{proof}

\begin{lem}
\label{Inegative}
Suppose that $E$ is indecomposable and $N(E)\le 0$. Then $\Aut(S)$ is a Jordan group.
\end{lem}

\begin{proof}
By Sect. \ref{maruyama}(iii),
 there is a nonnegative integer $r$ such that the group $\Aut_X(S)$ is isomorphic to the direct sum of $r$ copies of the additive group $k$. In particular, $\Aut_X(S)$ does not contain elements of finite order (except the identity element), because $\fchar(k)=0$. This implies that if $G \subset \Aut(S)$ is a group of finite order then it meets (the subgroup) $\Aut_X(S)$ only at the identity element. Now the exact sequence (2) tells us  that $G$ is isomorphic to its image $f(G) \subset \Aut(X)$. Since $\Aut(X)$ is a Jordan group, we are done.
\end{proof}

Lemmas \ref{Npositive},  \ref{DnotZero} and \ref{Inegative} leave us with the case of decomposable vector bundles $E$ with $N(E)=0$.

\section{Ruled surfaces and theta groups}
\label{theta}
\begin{sect}
Recall that every biregular automorphism $u$ of $\P^1$ that leaves invariant both {\sl zero} $(0:1)$ and {\sl infinity} $(1:0)$ is of the form
$$u: (a:b) \mapsto (a:\mu b) \ \forall (a:b) \in \P^1(k).$$
Here $\mu$ is a nonzero element of $k$ that may be described as follows. Let $\t_{\infty}$ be the one-dimensional tangent space to $\P^1$ at  $(0:1)$. Then the differential (tangent map) of $u$ at  $(0:1)$ acts on  $\t_{\infty}$ as multiplication by $\mu$.
This description gives us immediately the following elementary statement.
\end{sect}

\begin{prop}
\label{elemP}
Let $C$ be an algebraic curve over $K$ with two distinct points
$P_0, P_{\infty}\in C(k)$. Let $w: C \to \P^1$ be a biregular isomorphism
such that $w(P_0)=(0:1), \ w(P_{\infty})=(1:0)$.
Let $\varsigma$ be a biregular automorphism of $C$ that leaves invariant both $P_0$ and $P_{\infty}$.  The action of the differential (tangent map) of $\varsigma$ on the one-dimensional tangent space $\t_{P_{\infty}}(C)$ to $C$ at $P_{\infty}$ is multiplication by a nonzero constant that we denote by
$$\mu=\mu(C,\varsigma)=\mu(C,\varsigma; P_0,P_{\infty}).$$
Then the biregular automorphism $w \varsigma w^{-1}$ of $\P^1$ is defined by the formula
$$(a:b) \mapsto (a:\mu(C,\varsigma) b) \ \forall (a:b) \in \P^1(k).$$
In particular, if $\tilde{\varsigma}$ is a biregular automorphism of $C$ that leaves invariant both $P_0$ and $P_{\infty}$
and $\mu(C,\tilde{\varsigma}; P_0,P_{\infty})=\mu(C,\varsigma; P_0,P_{\infty})$ then $\tilde{\varsigma}=\varsigma$.
\end{prop}

\begin{sect}
\label{index}
Suppose that $E$ is {\sl decomposable}, i.e. $E$ is isomorphic to a direct sum $L_1 \oplus L_2$ of line bundles $L_1$ and $L_2$ on $X$. Then
$X=\P(E)$ is biregularly isomorphic (as a $\P^1$-bundle) to $\P(L_1 \oplus L_2)$. Further, we assume that $E=L_1 \oplus L_2$. As above, each $E_i$ gives rise to the section
$$s_i=t_{L_i}: X \to \P(L_1 \oplus L_2)=S.$$
The  images of  $s_1$ and $s_2$ in $S$ do not meet each other (Sect. \ref{maruyama}(iv)). On the other hand, each $\lambda \in k^{*}$ gives rise to the automorphism of $E$
$$i(\lambda): L_1 \oplus L_2 \to L_1 \oplus L_2$$
of $E$ that acts as the identity map on $L_1$ and as multiplication by $\lambda$ on $L_2$.
By abuse of notation, we continue to denote by $i(\lambda)$ the image
 $\proj(i(\lambda))$
 of $i(\lambda)\in \Aut(E)$ in $\Aut_X(S)$  defined in (2)  and view it as the automorphism of $S$.
Clearly, $\lambda$ is uniquely determined by the action of $i(\lambda)$ on any given fiber $C_z:=\pi^{-1}(z)$ of
 $\pi$ with $z\in X(k)$. Actually, one may reconstruct $\lambda$, using Proposition \ref{elemP}. Namely, let us consider two distinct $k$-points $P_{\infty}(z)=s_1(z)$ and $P_0(z)=s_2(z)$ on the curve $C_z$. Notice that $C_z$ is biregularly isomorphic to $\P^1$ and both $P_0(z)$ and $P_{\infty}(z)$ are fixed points of $i(\lambda)$. Now if we denote by $i(\lambda)_z:C_z \to C_z$ the biregular automorphism of $C_z$ induced by $i(\lambda)$ then one may easily check that
$$\lambda=\mu(C_z, i(\lambda)_z; P_0(z),P_{\infty}(z)).$$
\end{sect}

\begin{lem}
\label{key}
Suppose that $L_1$ and $L_2$ are line bundles on $X$, $E=L_1\oplus L_2$ and $S=\P(E)=\P(L_1 \oplus L_2)$. Let us consider the sections
$$s_1=t_{L_1}, \ s_2=t_{L_2}.$$
Suppose that an automorphism $\sigma\in \Aut_X(S)$ respects both images $s_1(X)$ and
$s_2(X)$, i.e.,
$$\sigma(s_1(X))=s_1(X), \  \sigma(s_2(X))=s_2(X).$$
Then there exists precisely one $\beta \in k^{*}$ such that $\sigma=i(\beta)$.
\end{lem}

\begin{proof}
Since $\sigma$ leaves invariant every fiber of $\pi$ and both $s_1$ and $s_2$ are sections of $\pi$,  all the points of both  $s_1(X)$ and $s_2(X)$ are fixed points of $\sigma$.  For each $z \in X(k)$ we keep the notation of Sect. \ref{index}.
Clearly, $\sigma$ induces the biregular automorphism $\sigma_z$ of $C_z$ that leaves invariant both $P_0(z)$ and $P_{\infty}(z)$.

Since both $L_1$ and $L_2$ are locally trivial, for each $x\in X(k)$ there is an open neighborhood $U_x \subset X$ such that the restrictions of both $L_1$ and $L_2$ to $U_x$ are trivial and we obtain a {\sl trivialization}, i.e, a biregular isomorphism
 $$\psi: \pi^{-1}(U_x) \cong U_x \times \P^1$$
of $\P^1$-bundles  $\pi^{-1}(U_x)\to U_x$  and the direct product $U_x \times \P^1\to U_x$ in such a way that $s_1(U_x)\subset \pi^{-1}(U_x)$ goes to the {\sl infinite} section $U_x \times (1:0)$ while the zero section
$s_2(U_x)\subset \pi^{-1}(U_x)$ becomes the {\sl zero} section $U_x \times (0:1)$. In particular, $\psi$ induces the biregular isomorphism $\psi_z: C_z \cong \P^1$ that sends  $P_0(z)$ to $(0:1)$ and $P_{\infty}(z)$ to $(1:0)$.

Using Proposition \ref{elemP}, we define
$$h(z)=\mu(C_z,\sigma_z; P_0(z),P_{\infty}(z)) \in k.$$
Clearly, $h(z)$ is a regular function on $X$ without zeros and therefore is a nonzero constant that we denote by $\beta$. So,
$$\beta=\mu(C_z,\sigma_z; P_0(z),P_{\infty}(z)).$$
However, according to the last formula of Section \ref{index},
$$\beta=\mu(C_z, i(\beta)_z; P_0(z),P_{\infty}(z)).$$
It follows from Proposition \ref{elemP} that
$$ i(\beta)_z=\sigma_z \ \forall z \in X(k).$$
In other words the actions of $i(\beta)$ and $\sigma$ coincide on $C_z(k)=\pi^{-1}(z)$ for all $z \in X(k)$.
Since the union of all $\pi^{-1}(z)$ coincides with $\pi^{-1}(X(k))=S(k)$, we conclude that  the actions of $i(\beta)$ and $\sigma$ coincide on $S(k)$, i.e.,  $\sigma=i(\beta)$.

\end{proof}

\begin{cor}
\label{indexTwo}
Suppose that $E=L_1\oplus L_2$ where $L_1$ and $L_2$ are mutually non-isomorphic line bundles of the same degree. 
Let $\Aut^{1}(S)$ be the subgroup of $\Aut(S)$ that consists of all automorphisms that leave invariant   both images $t_{L_1}(X)$ and  $t_{L_2}(X)$. Then $\Aut^{1}(S)$ is a normal subgroup in $\Aut(S)$ and its index is either $1$ or $2$.
\end{cor}

\begin{proof}
Let $s:X \to S$ be a section of $\pi$. If $\sigma \in \Aut(S)$ then
$$\sigma s f(\sigma)^{-1}: X \to X \to S \to S$$
is also a section of $\pi$, whose image in $S$ coincides with $\sigma s(X)$;
in addition, the self-intersection indices of $s$ and
$\sigma s f(\sigma)^{-1}$ do coincide. This implies that $\Aut(S)$  permutes the set of images of minimal sections of $\pi$.

 By Sect. \ref{maruyama}(iv), the set of all maximal line subbundles in $E$ consists of $L_1$ and $L_2$ and
  the set of all minimal sections consists of $t_{L_1}$ and $t_{L_2}$.
This implies that
$\Aut(S)$ permutes the elements of the two-element set $\{t_{L_1}(X), t_{L_2}(X)\}$. It follows that
  $\Aut^{1}(S)$ is the kernel of a certain group homomorphism from $\Aut(S)$ to the group of permutations in two letters. The rest is clear.
\end{proof}

\begin{sect}
\label{alg}
Let us fix a point $x_0 \in X(k)$. This provides $X$ with the structure of an abelian variety with $x_0$ being the zero of group law on $X$.

Suppose that $E$ is a direct sum $\II_X\oplus L$ of the trivial line bundle $\II_X$  and a line bundle $L$ over $X$. Let $\G(L)$ be the theta group attached to $L$ \cite[Sect. 23]{Mumford}. This means that $\G(L)$ is the algebraic $k$-group that is the group of biregular  automorphisms of the total space of $L$ that lift translations on $X$
and induce linear maps between the fibers of $L$. (Further we will identify $\G(L)$ with the group of its $k$-points.)
It is known that $\G(L)$ sits in a short exact sequence of groups
$$\{1\} \to k^{*} \to \G(L) \stackrel{g}{\to} H(L) \to \{1\}  \eqno{(3)}$$
where each $\alpha \in k^{*}$ acts as  natural multiplicaton by $\alpha$ and
$$H(L)=\{z\in X(k)\mid \TT_z^{*}L \cong L\}$$
is a subgroup of $X(k)$. In addition, every $u \in \G(L)$ is a lift to $L$  of the translation
 $\TT_{g(u)}: X \to X$ by  $g(u)\in X(k)$.  Let us construct the group embedding
$$i_L:\G(L) \to \Aut(S),$$
 whose restriction to $k^*$ coincides with $i: k^* \to \Aut(S)$.  Let $u \in \G(L)$ and $y:=g(u)\in X(k)$. Let us define a biregular automorphism $\tilde{i}_L(u)$ of $E$ that is defined as follows.
For each $x \in X(k)$ the map $\tilde{i}_L(u)$ sends the fiber
$E_x=\{x \times \AA^1\} \times L_x$ to the fiber $E_{x+y}=\{(x+y) \times \AA^1\} \times L_{x+y}$ by the formula
$$((x,a), l) \mapsto ((x+y,a),u(l)) \ \forall a\in k, \ l \in E_x.$$
By definition,
$$\tilde{i}_L(\II_X\oplus \{0\})=\II_X\oplus \{0\}, \ \tilde{i}_L(\{0\}\oplus L)=\{0\}\oplus L . \eqno{(4)}$$
Clearly, $\tilde{i}_L(u):E \to E$ lifts  $\TT_y:X \to X$. Since $\tilde{i}_L(u)$ is linear, one may define its projectivization
$$i_L(u): \P(E) \to \P(E),$$
which is a biregular automorphism of $ \P(E)$ that lifts
$$\TT_y \in X(k)\subset \Aut(X).$$
 It is also clear that for each $\alpha \in k^*$ the automorphism $i_L(\alpha)$ coincides with $i(\alpha)$ and the map
$$\G(L) \to \Aut(S), \ u \mapsto i_L(u)$$
is a group embedding. Further we will identify $\G(L)$ with its image in $\Aut(S)$. Clearly,
$$f(\G(L))=g(\G(L))=H(L) \subset X(k) \subset \Aut(X).$$
It follows from (4) that each $u \in \G(L)$ leaves invariant the images $t_{L_1}(X)$ and  $t_{L_1}(X)$ of $t_{L_1}$ and $t_{L_2}$.

Now assume that $\deg(L)=0$. Then $L$ is algebraically equivalent to the trivial line bundle and therefore
  $H(L)=X(k)$ and the group $\G(L)$ is commutative \cite[Sect. 23]{Mumford}. In particular,  using (3), we get the short exact sequence of {\sl commutative} groups
$$\{1\} \to k^* \to \G(L) \stackrel{g}{\to} X(k) \to \{1\} \eqno{(5)}.$$
This implies that
$$f(\G(L))=g(\G(L))=X(k) \subset f(\Aut(S))\subset \Aut(X)  \eqno{(6)}.$$
See \cite[Lemma 8 on p. 99--100]{Mar2} for an explicit description of  $f(\Aut(S))$.
\end{sect}

The following assertion may be viewed as a rewording of \cite[Th. 3(2) on pp. 106--107]{Mar2}.

\begin{lem}
\label{comm}
Suppose that $E=\II_X\oplus L$ where $L$ is a degree zero line bundle on $X$ that is not isomorphic to $\II_X$.
Then $\G(L)$ is a commutative subgroup of finite index in $\Aut(S)$.
\end{lem}

\begin{proof}
It follows from (6) that
$$\G(L)\subset f^{-1}(X(k))\subset \Aut(S).$$
Since $X(k)$ is a subgroup of finite index in $\Aut(X)$, the preimage
$\HH:=f^{-1}(X(k))$ is a subgroup of finite index in $\Aut(S)$. So, it suffices to check that
$\G(L)$ is a subgroup of finite index in $\HH$. Let us put
$$\HH^1:= \HH\bigcap \Aut^1(S)=\{\sigma \in \HH\mid \sigma (t_{\II_X}(X))= t_{\II_X}(X),
 \ \sigma(t_L(X))=t_L(X)\}.$$
Clearly, $\HH^1$ is a subgroup of $\HH$ that contains $\G(L)$. It follows from Corollary \ref{indexTwo} that either $\HH^1=\HH$ or $\HH^1$ is a subgroup of index $2$ in $\HH$. However, $\HH^1$ is a subgroup of
 finite index  in $\HH$.
I claim that $\G(L)=\HH^1$. In order to prove that,
first notice that since $[\HH:\HH^1]=1$ or $2$,
$$2 \cdot X(k) \subset f(\HH^1)\subset X(k).$$
Since $k$ is algebraically closed, the group $X(k)$ is divisible; in particular, $2\cdot X(k)=X(k)$ and therefore
$$ f(\HH^1)= X(k)=f(\G(L)).$$
This implies that for each $\sigma \in \HH^1$ there exists $u \in \G(L)$ such that
$$f(\sigma)=f(u) \in X(k).$$
Recall that both $\sigma$ and $u$ leave invariant both $t_{\II_X}(X)$ and $t_L(X)$.
Then the automorphism
$\tau:=u^{-1}\sigma$ lies in $\ker(f)=\Aut_S(X)$ and also leaves invariant both $t_{\II_X}(X)$ and $t_L(X)$.
This implies that
$$\tau (t_{\II_X}(X))=t_{\II_X}(X), \ \tau (t_L(X))= t_L(X).$$
By Lemma \ref{key}, there exists $\beta \in k^{*}$ such that
$$\tau=i(\beta) \in \Aut(S).$$
However,
$$i(\beta)=i_L(\beta) \subset \G(L).$$
This implies that
$$\sigma=u (u^{-1}\sigma)=u \tau=u \ i(\beta) \in \G(L)$$
and we are done.
\end{proof}

Lemma \ref{comm} implies readily the following assertion.

\begin{cor}
\label{Jalg}
Suppose that $E=\II_X\oplus L$ where $L$ is a degree zero line bundle on $X$ that is not isomorphic to $\II_X$.
Then $\Aut(S)$ is a Jordan group.
\end{cor}

\section{Proof of main result}
\label{final}
\begin{proof}[Proof of Theorem \ref{elliptic}]
As was explained in Remark \ref{prelimiv}, it suffices to prove that
 $\Aut(S)$ is Jordan when $S=\P(E)$ is an elliptic ruled surface over $X$.

If $E$ is indecomposable then the result follows from  Lemmas \ref{Npositive} and \ref{Inegative}.

Suppose that $E$ is decomposable. If $N(E) \ne 0$ then the result follows from Lemma \ref{DnotZero}. Now assume that $N(E)=0$. This implies that $E$ is isomorphic to a direct sum $L_1\oplus L_2$ of line bundles $L_1$ and $L_2$ over $X$. By tensoring $E$ by $\L=L_1^{-1}$, we may and will assume that $E=\II_X \oplus L$. We have $\det(E)=\deg(L)$.
 In addition,  the degree $d$  of a maximal line subbundle of $E$ is {\sl nonnegative}, because $\deg(\II_X)=0$.
I claim that $\deg(L)= 0$. Indeed, if $\deg(L)>0$ then  $d \ge \deg(L)>0$ and therefore
$$N(E)=\deg(E)-2d=\deg(L)-2d\le d-2d=-d<0,$$
i.e., $N(E)<0$, which is not the case. If $\deg(L)<0$ then $\deg(E)=\deg(L)<0$ and
$$N(E)=\deg(E)-2d \le \deg(E)<0,$$
i.e., $N(E)<0$, which is not the case. So, $\deg(L)=0$. If $L$ is isomorphic to $\II_X$ then
$$S=\P(\II_X \oplus L)=\P(\II_X \oplus \II_X)=\P(\II_X^2)=\P(X \times \AA^2)=X\times \P^1.$$
In this case it is known \cite{Mar2} that
$$\Aut(X\times \P^1)=\Aut(X) \times \Aut(\P^1)=\Aut(X) \times \PGL(2,k):$$
it is a product of two Jordan groups and therefore is also Jordan \cite{Popov}.
If $L$ is {\sl not} isomorphic to $ \II_X$, Corollary \ref{Jalg} implies that $\Aut(S)$ is Jordan.

\end{proof}

\section{Group theory}
\label{group}
\begin{sect}
\label{pairing}
Let
$$\{1\} \to \A \to \B \stackrel{j}{\to} \C \to \{1\}$$
be a short exact sequence of groups such that $\C$ is commutative and $\A$ is a {\sl central} subgroup of $\B$.
We denote by $1_{\A}$ the identity element of $\A$. We write down the group law on $\C$ additively and on $\A$ and $\B$ multiplicatively.

It is well known (see, e.g. \cite{ZarhinMZ74}) that there is an alternating bimultiplicative pairing
$$e: \C \times \C \to \A, \ (j(b_1), j(b_2)) \mapsto b_1 b_2 b_1^{-1} b_2^{-1}.$$
Clearly, $b_1$ and $b_2$ commute if and only if $e(j(b_1),j(b_2))=1_{\A}$. In particular,
the pairing $e$ is trivial, i.e.,
$$e(x,y)=1_{\A} \ \forall x,y \in \C$$
if and only if $\B$ is commutative.

Suppose that $\A$ has finite exponent. This means that there is a positive integer $r$ such that
$$a^r=1_{\A} \ \forall a \in \A.$$
This implies that
$$1_{\A}=e(x,y)^r=e(x^r,y)=e(x,y^r) \ \forall x,y \in \C.$$
It follows that for each $b_1,b_2 \in \B$
$$e(j(b_1^r), j(b_2^r))=e(j(b_1),j(b_2))^{r^2}=1_{\A},$$
i.e., $b_1^r$ and $b_2^r$ commute.
Let $\B_r$ be the subgroup of $\B$ generated by all $r$th powers $b^r$ ($b\in \B$). Clearly, $\B_r$ is a {\sl normal commutative} subgroup in $\B$ and
$$j(\B_r)=r\C=\{rc \mid c\in \C\}.$$

 \end{sect}
\begin{lem}
\label{abext}
Suppose that $Y$ is an abelian variety over an algebraically closed field $\kappa$ of arbitrary characteristic and $\Delta$ is a finite group. Suppose that a group $G$ sits in an exact sequence
$$\{1\} \to \Delta \to G \stackrel{f}{\to} Y(\kappa).$$
Then $G$ is Jordan group.
\end{lem}

\begin{proof}
We write down the group law on (commutative) $Y(k)$ additively and on $\Delta$ and $G$ multiplicatively.
We may assume that $g:=\dim(Y)>0$, since otherwise $G=\Delta$ is finite and we are done.  Recall \cite{Mumford} that if $m$ is a positive integer then the kernel $Y_m$ of multiplication by $m$ in $Y(\kappa)$ is a direct sum of, at most, $2g$ cyclic groups. This implies
that if $H$ is a finite subgroup of $Y(\kappa)$ then $H$ is isomorphic to a direct sum of, at most, $2g$ cyclic groups. It follows  that for each positive integer $n$ the index $[H:nH]$ divides $n^{2g}$; in particular, it does not exceed the universal constant $n^{2g}$, which does not depend on $H$.
Let $r$ be the exponent of $\Delta$ and $d$ the order of its automorphism group $\Aut(\Delta)$.

Let $B$ be a finite subgroup in $G$. Let $B_0$ be the normal subgroup of $B$ generated by $d$th powers of elements of $B$. Clearly, $B_0$ is normal in $B$. On the other hand, all elements of $B_0$ commute with $\Delta$. This implies that the intersection
$$\Delta_0:=\Delta\bigcap B_0$$
lies in the center of $\Delta$. In particular, $\Delta_0$ is commutative. Clearly, $\Delta_0$ lies in the center of $B_0$ and is a normal subgroup of $B_0$. The group $B_0$ sits in the short exact sequence
$$\{1\} \to \Delta_0 \subset B_0 \stackrel{f_0}{\to}  d \ f(B) \to \{0\}$$
with central subgroup $\Delta_0$ and commutative
$d \ f(B)$.
(Here
$$f_0: B_0 \to  Y(k)$$
is the restriction of $f$ to $B_0$.)
This implies that the index
$$[B:B_0] =\frac{\#(\Delta)}{\#(\Delta_0)} \cdot \frac{\#(f(B))}{\#(d\ f(B))}\le \#(\Delta)\cdot d^{2g}.$$
Let us put
$$\A=\Delta_0, \B=B_0, \C=f(B_0)=d \ f(B), \ j=f_0: B_0 \twoheadrightarrow d\ f(B)=\C.$$
Then we get the short exact sequence as in Sect. \ref{pairing}:
$$\{1\} \to \A \to \B \stackrel{j}{\to} \C \to \{1\}.$$
Since $r$ is the exponent of $\Delta$ and $\A=\Delta_0$ is a subgroup of $\Delta$,
$$a^r=1_{\A} \ \forall a \in \A=\Delta_0.$$
Let $\B_r$ be the subgroup of $\B=B_0$ that is generated by $r$th powers of all elements of $B_0$. Clearly, $\B_r$ is normal in $B$. Thanks to the arguments of Sect. \ref{pairing}, $\B_r$ is {\sl commutative} and
$$j(\B_r)=f_0(\B_r)=f(\B_r)= r \ f(B_0)=rd \ f(B).$$
It follows easily that $[B:\B_r]$ divides $(rd)^{2g} \#(\Delta)$.
\end{proof}

\begin{cor}
\label{abFinite}
Let $X$ be an elliptic curve over an algebraically closed field $\kappa$ of arbitrary characteristic and $\Delta$ is a finite group. Suppose that a group $G$ sits in an exact seqwuence
$$\{1\} \to \Delta \to G \stackrel{f}{\to} \Aut(X).$$
Then $G$ is Jordan.
\end{cor}

\begin{proof}
Fix a point $x_0 \in X(\kappa)$. Then $X$ becomes the one-dimensional abelian variety with $x_0$ be the zero of group law.
Then by (0) we may view
$X(\kappa)$ as a subgroup in $\Aut(X)$ of finite index.
This implies that $G_1:=f^{-1}(X(\kappa))$ is a subgroup of finite index in $G$. Therefore, it suffices to check that $G_1$ is Jordan. However, $G_1$ sits in an exact sequence
$$\{1\} \to \Delta \to G_1 \stackrel{f}{\to} X(\kappa).$$
Now Lemma \ref{abext} implies that $G_1$ is Jordan.
\end{proof}

\begin{rem}
One may deduce Lemma \ref{abext} from Lemma 2.8 in \cite{ProSh}.
\end{rem}

\end{document}